\documentclass{article}
\usepackage{amssymb}
\usepackage{amsthm}
\usepackage[curve,matrix,arrow]{xy}
\theoremstyle{plain}\newtheorem{Theorem}{Theorem}[section]
\theoremstyle{plain}
\theoremstyle{plain}
\theoremstyle{plain}
\theoremstyle{plain}
\theoremstyle{plain}
\theoremstyle{definition}
\theoremstyle{definition}
\theoremstyle{definition}
\theoremstyle{definition}\newtheorem{Remark}[Theorem]{Remark}
\theoremstyle{definition}

\def\add{\mathrm{add}}

\def\coMack{\mathrm{coMack}}    
    
           \def\tenkP{\otimes_{kP}}
\def\ddim{\mathrm{ddim}}     

\def\End{\mathrm{End}}           \def\tenkQ{\otimes_{kQ}}

\def\Ext{\mathrm{Ext}}

\def\Hom{\mathrm{Hom}}

             \def\tenA{\otimes_A}
             \def\tenB{\otimes_B}

\def\mod{\mathrm{mod}}

\def\op{\mathrm{op}}

\def\Res{\mathrm{Res}}

\title{The dominant dimension of cohomological Mackey functors} 
\author{Markus Linckelmann} 
\date{}

\begin{document}

\maketitle

\begin{abstract}
We show that a separable equivalence between symmetric algebras
preserves the dominant dimensions of certain endomorphism algebras of
modules. We apply this to show that the dominant dimension of the
category $\coMack(B)$ of cohomological Mackey functors of a $p$-block 
$B$ of a finite group with a nontrivial defect group is $2$.
\end{abstract}

\section{Introduction}

Let $k$ be a field. 
Following Tachikawa \cite{Tac}, the {\it dominant dimension of a 
finite-dimensional $k$-algebra $A$}, which we will denote by 
$\ddim(A)$, is the largest nonnegative integer $d$ such that there 
exists an injective resolution
$$\xymatrix{0\ar[r] & A\ar[r] & I^0\ar[r]&I^1 \ar[r]&\cdots}$$
of $A$ as a right $A$-module with the property that $I^n$ is projective 
for $0\leq n\leq d-1$, provided there is such an integer. If for any
injective resolution $I$ of $A$ the term $I^0$ is not projective, then 
$\ddim(A)=$ $0$, and if there exists an injective resolution $I$ of $A$
such that $I^n$ is projective for all $n\geq$ $0$, then we adopt the 
convention $\ddim(A)=$ $\infty$. By a result of M\"uller 
\cite[Theorem 4]{Mue} the dominant dimension is equal to the obvious 
left module analogue. In order to calculate $\ddim(A)$ it suffices to
consider a minimal injective resolution of $A$ as a right $A$-module.

If $k$ has prime characteristic $p$ and if $A$ is a source algebra of 
a block $B$ of a finite group algebra $kG$ with a defect group $P$, 
then, by the source algebra version \cite[Theorem 1.1]{LincoMack} of 
a result of Yoshida in \cite{YoshidaII}, the category $\coMack(B)$ of 
cohomological Mackey functors of $G$ associated with $B$ is equivalent 
to the right module category of the endomorphism algebra  $E=$
$\End_A(\oplus_{Q\leq P}\ A\tenkQ k)$. The {\it dominant dimension of 
$\coMack(B)$} is defined as the dominant dimension of $E$.

\begin{Theorem} \label{ddimcoMack}
Suppose that $k$ has prime characteristic $p$. Let $G$ be a finite 
group and $B$ a block of $kG$ with a nontrivial defect group $P$.  
The dominant dimension of $\coMack(B)$ is equal to $2$.
\end{Theorem}

A finite-dimensional $k$-algebra $A$ is called {\it symmetric} if $A$ 
is isomorphic to its $k$-dual $A^*$ as an $A$-$A$-bimodule. 
Given two symmetric $k$-algebras $A$ and $B$, we say
that an $A$-$B$-bimodule $M$ {\it induces a separable equivalence
between $A$ and $B$} if $M$ is finitely generated projective as a left 
$A$-module and as a right $B$-module such that $A$ is isomorphic to a 
direct summand of $M\tenB M^*$ as an $A$-$A$-bimodule and $B$ is 
isomorphic to a direct summand of $M^*\tenA M$ as a $B$-$B$-bimodule. 
In that case we say that $A$ and $B$ are {\it separably equivalent}. 
See \cite{LinHecke} for details and more general versions of this 
notion.
We will show that Theorem \ref{ddimcoMack} is an immediate consequence 
of results from \cite{LincoMack}, \cite{LR1}, \cite{Mue}, and the 
following.

\begin{Theorem} \label{endoddim}
Let $A$, $B$ be symmetric $k$-algebras. Let $M$ be an 
$A$-$B$-bimodule inducing a separable equivalence between $A$ and $B$.
Let $U$ be a finitely generated $A$-module such that $A$ is isomorphic
to a direct summand of $U$, and let $V$ a finitely generated 
$B$-module such that $B$ is isomorphic to a direct summand of $V$. 
Suppose that $M^*\tenA U\in$ $\add(V)$ and that $M\tenB  V\in$ 
$\add(U)$. Then the dominant dimensions of $\End_A(U)$ and of 
$\End_B(V)$ are equal.
\end{Theorem}

\section{Proof of Theorem \ref{endoddim}}

Let $A$ be a finite-dimensional $k$-algebra, and let $U$, $V$ be 
finitely generated $A$-modules. 
We use without further reference the following standard facts. 
If $V$ belongs to $\add(U)$, then $\Hom_A(U,V)$ is a projective 
right $\End_A(U)$-module, and any finitely generated projective right
$\End_A(U)$-module is of this form, up to isomorphism.  Given two 
idempotents $i$, $j$ in $A$, every homomorphism of right $A$-modules 
$iA\to$ $jA$ is induced by left multiplication with an element in $jAi$. 
Multiplication by $i$ is exact; in particular, if $Z$ is a complex of 
$A$-modules which is exact or which has homology concentrated in a
single degree, the same is true for finite direct sums of complexes
of the form $iZ$. Translated to endomorphism algebras this implies that 
for any two $A$-modules $V$, $W$ in $\add(U)$, any homomorphism of right 
$\End_A(U)$-modules $\Hom_A(U,V)\to$ $\Hom_A(U,W)$ is induced by 
composition with an $A$-homomorphism $V\to$ $W$. Thus any complex of 
finitely generated projective right $\End_A(U)$-modules is isomorphic to 
a complex obtained from applying the functor $\Hom_A(U,-)$ to a complex 
of $A$-modules $Z$ whose terms belong to $\add(U)$. Moreover, if 
$\Hom_A(U,Z)$ is exact, then so is any complex of the form 
$\Hom_A(U',Z)$, where $U'\in$ $\add(U)$.
We use further the well-known fact that if $A$, $B$ are symmetric
algebras and if $M$ is an $A$-$B$-bimodule which is finitely generated
projective as a left $A$-module and as a right $B$-module, then the
functors $M\tenB-$ and $M^*\tenA-$ between $\mod(A)$ and $\mod(B)$
are biadjoint.

\begin{proof}[{Proof of Theorem \ref{endoddim}}] 
The following argument from the proof of \cite[Theorem 3.1]{LincoMack} 
shows that $\add(M\tenB V)=$ $\add(U)$ and $\add(M^*\tenA U)=$ 
$\add(V)$.
By the assumptions, we have $\add(M\tenB V)\subseteq$ $\add(U)$.
Thus $\add(M^*\tenA M \tenB V)\subseteq$ $\add(M^*\tenA U)\subseteq$
$\add(V)$. Since $B$ is isomorphic to a direct summand of the
$B$-$B$-bimodule $M^*\tenA M$, it follows that $V$ is isomorphic
to a direct summand of $M^*\tenA M\tenB V$. Thus the previous
inclusions of additive categories are equalities. The same
argument with reversed roles shows the second equality.

Set $E=$ $\End_A(U)$ and $F=$ $\End_B(V)$. 
By the assumptions, $\add(U)$ contains all finitely generated
projective $A$-modules; similarly for $\add(V)$. In particular,
if $U'$ is a projective $A$-module, then $\Hom_A(U,U')$ is
a projective right $E$-module, and by \cite[Proposition 3.2]{LR1},
$\Hom_A(U,U')$ is also an injective right $E$-module.  Let
$$\xymatrix{0\ar[r] & E\ar[r] & I^0\ar[r]&I^1 \ar[r]&\cdots}$$
be an injective resolution of $E$ as a right $E$-module.
Suppose that there is a positive integer $d$ such that $I^n$ is 
projective for $0\leq$ $n\leq$ $d-1$. We will show that
there is an injective resolution
$$\xymatrix{0\ar[r] & F\ar[r] & J^0\ar[r]&J^1 \ar[r]&\cdots}$$
of $F$ as a right $F$-module such that $J^n$ is projective
for $0\leq$ $n\leq$ $d-1$. Since the right $E$-modules
$I^n$ are projective and injective for $0\leq$ $n\leq$ $d-1$, it 
follows from \cite[Proposition 3.2]{LR1} that there are finitely 
generated projective $A$-modules $U_n$ for $0\leq$ $n\leq$ $d-1$ such 
that the sequence
$$\xymatrix{0\ar[r] & E\ar[r] & I^0\ar[r]&I^1 \ar[r]&\cdots\ar[r]
& I^{d-1}}$$
is isomorphic to a sequence of the form
$$\xymatrix{0\ar[r] & \Hom_A(U,U)\ar[r] & \Hom_A(U,U_0)\ar[r]
&\Hom_A(U,U_1) \ar[r]&\cdots\ar[r] & \Hom_A(U,U_{d-1})}$$
which is obtained from applying the functor $\Hom_A(U,-)$ to
a sequence of $A$-modules 
$$\xymatrix{0\ar[r] & U\ar[r] & U_0\ar[r]&U_1 \ar[r]&\cdots\ar[r]
& U_{d-1}}$$
It follows from the remarks at the beginning of this section that
for any $A$-module $U'$ in $\add(U)$, applying the functor
$\Hom_A(U',-)$ to the previous sequence of $A$-modules yields
an exact sequence of right $\End_A(U')$-modules of the form
$$\xymatrix{0\ar[r] & \Hom_A(U',U)\ar[r] & \Hom_A(U',U_0)\ar[r]
&\Hom_A(U',U_1) \ar[r]&\cdots\ar[r] & \Hom_A(U',U_{d-1})}$$
By the assumptions, the $A$-module $U'=$ $M\tenB V$ belongs
to $\add(U)$. Thus we obtain an exact sequence of the form
$$\xymatrix{0\ar[r] & \Hom_A(M\tenB V,U)\ar[r] 
& \Hom_A(M\tenB V,U_0)\ar[r] & \cdots\ar[r] 
& \Hom_A(M\tenB V,U_{d-1})}$$
Since $M\tenB-$ is left adjoint to $M^*\tenA-$, it follows that
the previous exact sequence is isomorphic to an exact sequence
of the form
$$\xymatrix{0\ar[r] & \Hom_B(V,M^*\tenA U)\ar[r] 
& \Hom_B(V,M^*\tenA U_0)\ar[r] 
& \cdots \ar[r] & \Hom_B(V,M^*\tenA U_{d-1})}$$
Since the $A$-modules $U_n$ are projective for $0\leq$ $n\leq$ $d-1$,
it follows that the $B$-modules $M^*\tenA U_n$ are projective as
well, hence in $\add(V)$. By \cite[Proposition 3.2]{LR1}, the 
projective right $F$-modules $\Hom_B(V,M^*\tenA U_n)$ are therefore 
also injective. Thus the preceding sequence is the beginning of an 
injective resolution of the right $F$-module $\Hom_B(V,M^*\tenA U)$ 
which has the property that its first $d$ terms are projective. 
Since $V$ belongs to $\add(M^*\tenA U)$, it follows
that $F=$ $\Hom_B(V,V)$ is isomorphic, as a right $F$-module, to a 
direct summand of a direct sum of finitely many copies of the right
$F$-module $\Hom_B(V,M^*\tenA U)$. This implies that $F$ has an 
injective resolution as a right $F$-module whose first $d$ terms
are projective. This shows that $\ddim(F)\geq$ $\ddim(E)$
(including the case where both are $\infty$). Exchanging the roles
of $A$, $E$ and $B$, $F$, respectively, yields the result.
\end{proof}

\section{Proof of Theorem \ref{ddimcoMack}}

We suppose in this section that $k$ is a field of prime characteristic
$p$. Let $G$ be a finite group and $B$ a block of $kG$ with a
nontrivial defect group $P$. Let $i\in$ $B^P$ be a source idempotent
of $B$ and set $A=$ $iBi$; that is, $A$ is a source algebra of $B$.
Note that $A$, $B$, $kP$ are symmetric algebras. It is well-known that 
$A$ and $kP$ are separably equivalent via the bimodule $A_{kP}$ and its
dual, which is isomorphic to ${_{kP}A}$; see e. g. 
\cite[Proposition 4.2]{Liperm} for a proof (the hypothesis on $k$
being algebraically closed in that paper is not needed for this
result). 

\begin{proof}[Proof of Theorem \ref{ddimcoMack}]
With the notation above, set $U=$ $\oplus_Q\ A\tenkQ k$, where $Q$ runs 
over the subgroups of $P$, and set $E=$ $\End_A(U)$.
By \cite[Theorem 1.1]{LincoMack} we have $\coMack(B)\cong$
$\mod(E^\op)$. Set $V=$ $\oplus_Q\ kP\tenkQ k$, where $Q$ runs as
before over the subgroups of $P$, and set $F=$ $\End_{kP}(V)$.
As in the proof of \cite[Theorem 1.6]{LincoMack}, the functor 
$A\tenkP-$ sends $V$ to $\add(U)$ and the functor 
${_{kP}A}\tenA-$ sends $U$ to $\add(V)$, because $A$ has a 
$P\times P$-stable $k$-basis, hence preserves the classes of 
$p$-permutation modules. By definition, the dominant dimension of 
$\coMack(B)$ is equal to $\ddim(E)$.
It follows from Theorem \ref{endoddim} that $\ddim(E)=$ $\ddim(F)$.
We have $\ddim(F)\geq$ $2$ by the general Morita-Tachikawa
correspondence. Since the argument is very short, we sketch it (the
dual argument for left modules is in the proof of \cite[Theorem 2]{Mue},
for instance). Let $0\to$ $V\to$ $I^0\to$ $I^1\to\cdots$ be an 
injective resolution of $V$ as a $kP$-module, with $I^n$ finitely 
generated for all $n\geq$ $0$. The modules $I^n$ are also projective 
since $kP$ is symmetric. Applying $\Hom_{kP}(V,-)$ yields an exact 
sequence 
$$\xymatrix{0\ar[r] & \Hom_{kP}(V,V)=F\ar[r] 
& \Hom_{kP}(V,I^0)\ar[r] & \Hom_{kP}(V, I^1)}$$
of right $F$-modules.
The last two terms are projective as right $F$-modules because 
the finitely generated injective $kP$-modules $I^0$ and $I^1$ are in 
$\add(V)$, and the last two terms are also injective right $F$-modules 
by \cite[Proposition 3.2]{LR1} or by \cite[(17.2)]{Mor}.

By M\"uller's Lemma 3 in \cite{Mue}, in order to show that $\ddim(F)=$ 
$2$, it suffices to show that $\Ext^1_{kP}(V,V)$ is nonzero. The 
summand of $V$ indexed by $P$ is the trivial $kP$-module. Since $P$ is 
nontrivial, it follows that $\Ext^1_{kP}(k,k)\neq$ $\{0\}$, and hence 
$\Ext^1_{kP}(V,V)\neq$ $\{0\}$. Theorem \ref{ddimcoMack} follows.
\end{proof}

\begin{Remark} 
One can prove Theorem \ref{ddimcoMack} also without using Theorem
\ref{endoddim}, by showing directly that $\Ext_A^1(U,U)$ is nonzero, 
and then applying \cite[Lemma 3]{Mue} as in the proof above. Indeed, 
the separable equivalence between $A$ and $kP$ implies that $kP$ is 
isomorphic to a direct summand of $A$ as a $kP$-$kP$-bimodule. Thus 
$A\tenkP k$ has a trivial summand as a left $kP$-module. Note that 
$U\cong$ $A\tenkP V$. A standard adjunction implies that we have an 
isomorphism $\Ext_A^*(U,U)\cong$ $\Ext^*_{kP}(V, \Res^A_{kP}(U))$. By 
the above, both $V$ and $\Res^A_{kP}(U)$ have a trivial summand, and 
hence $\Ext^*_{kP}(k,k)$ is a summand of $\Ext^*_A(U,U)$ as a graded 
$k$-vector space. In particular, $\Ext^1_A(U,U)$ is nonzero since 
$\Ext^1_{kP}(k,k)$ is nonzero. 
\end{Remark}


\end{document}